\documentclass[11pt]{article}
\usepackage{graphicx}
\usepackage{float}
\usepackage{subfig}
\usepackage{overpic}
\usepackage{amsthm,amsmath,amssymb}
\usepackage{mathrsfs}
\usepackage{amsmath}
\usepackage[noend]{algpseudocode}
\usepackage{algorithmicx,algorithm}

\newtheorem{theorem}{Theorem}[section]

\newtheorem{coro}{Corollary}[section]

\newtheorem{lemma}{Lemma}[section]

\author{Pengyang Fan, Chao Zhai~\thanks{Pengyang Fan and Chao Zhai are with School of Automation, China University of Geosciences, Wuhan 430074 China, and with Hubei Key Laboratory of Advanced Control and Intelligent Automation for Complex Systems and Engineering Research Center of Intelligent Technology for Geo-exploration, Ministry of Education, Wuhan 430074 China. Corresponding author: Chao Zhai
(email: zhaichao@amss.ac.cn).}}

\date{August 2022}

\begin{document}

\title{Distributed Control Strategy for Layered Barrier Coverage of Multi-Agent Systems in Uncertain Environments}

\maketitle

\begin{abstract}
This paper presents a distributed multi-layer ring barrier coverage algorithm. In order to achieve single-layer ring barrier coverage, a distributed single-layer ring barrier coverage algorithm that maximises the probability of monitoring is proposed. Considering the security risks of single-layer barrier coverage, a distributed adjustment mechanism between multiple layers of barriers is designed and combined with the single-layer ring barrier coverage algorithm to propose a distributed multi-layer ring barrier coverage algorithm. Furthermore, we present a theoretical analysis of the proposed algorithm to demonstrate its effectiveness and necessity. Finally, our algorithm is verified by numerical simulation and experiment.
\end{abstract}

\section{Introduction}

Multi-agent systems(MASs) are composed of agents that interact with each other in an environment.
Each agent is a system, MASs are systems in which a large number of agents are grouped together and realise an overall behaviour or activity.
Agents can be natural creatures\cite{VT1995}, artificial robots or mobile sensors\cite{ro2020}.
The MASs aims to take a distributed approach to solve some large and complex problems.
Each agent is an independent individual that can perceive the environment, process information, communicate, learn, and make decisions independently.
Since each agent adopts an independent strategy, coordinated control of multi-agent systems is essential in order for the whole system to accomplish a common goal, and this area has attracted many scholars to conduct research.

Multi-agent coverage control is a hot research topic in multi-agent coordination control.
Multi-agent coverage control refers to a group of agent bodies with mobile, communication, computing, and learning capabilities to sense the environment and perform a given task in a distributed manner in a given or indefinite area, such as: search and rescue, missile interception, monitoring, sweeping, etc\cite{zhaim2021}.
Multi-agent coverage control can be classified into area coverage, sweeping coverage and barrier coverage according to the area covered by the agent.
Area coverage is a series of operations in which each agent body determines its optimal state in the area through communication, computation, and coordination and achieves this state through some control science methods.
The most classic one is the multi-agent coverage algorithm based on Voronoi partition\cite{CJ2004}, in which each agent divides a convex region into sub-regions through communication, and each agent uses a strategy of moving to the center of mass of the sub-region to maximize the coverage quality. This method can cover a convex region to the maximum extent.
On this basis, many scholars have found many problems and proposed some solutions.
For example, to solve the non-convex region, some scholars proposed the Voronoi center-of-mass coverage algorithm for non-convex region based on the geodesic Voronoi partition algorithm\cite{THM2013}; to solve the time-varying density function problem, some scholars proposed the Voronoi center-of-mass coverage in dynamic environment based on the control barrier function\cite{SM2019}; to solve the coverage problem in uncertain environment, some scholars proposed the Voronoi center-of-mass coverage in uncertain environment based on the Bayesian estimation\cite{BA2020}.
Sweep coverage not only requires the agent to reach the designated area but also requires agent to be able to traverse the entire area to achieve cleaning of the environment.
For example, some scholars have achieved equal-task sweep coverage of a class of regions based on equal-task partitioning methods, which can improve the overall efficiency\cite{zhai13}. Some scholars have also proposed a multi-agent sweeping coverage algorithm based on the temperature field approach\cite{IS2016}.
Moreover, literature \cite{Zheng2022} combines Voronoi segmentation with a temperature field approach to design a distributed overlay method that enables each agent to have the same workload.

Multi-agent barrier coverage refers to the coverage of a group of agents on a line, which is usually used to monitor whether a creature or object crosses the line or to intercept objects that attempt to cross on the line.
The literature \cite{ks2005} proposes the definition of barrier coverage and k-barrier coverage.
The k-barrier coverage is further divided into weak k-barrier coverage and strong k-barrier coverage\cite{ks2005}\cite{liu2008}\cite{wang2013}.
The strong k-barrier coverage means that an intruder is detected by at least k agents regardless of any path into or through the target area.
Besides, Chen et al. proposed the concept of local barrier coverage, which can reduce the number of agents compared to global fence overlays and can also be used for general cases\cite{chen2007}.
However, local barrier coverage is a security risk, as intruders can potentially traverse the area without being detected.
There are currently many algorithms for barrier coverage and k-barrier coverage.
For example, the coverage-based approach proposed in literature \cite{zhai16} can equip the task of assigning intrusion probability to intercept the intruded items thus achieving protection of the target.
In addition to this, scholars have designed a distributed algorithm that can achieve a uniform barrier coverage between two landmarks\cite{cheng2009}.
Ban et al. investigate the strong k-barrier coverage problem of mobile sensor networks over open belt using a grid-based approach in \cite{banD2010}.
However, the algorithm can only achieve coverage on a straight line between two points, and cannot achieve coverage on a curve, nor can it achieve coverage on a closed curve.
In practical applications, if the targets in the area need to be protected or monitored in an all-round way, the whole boundary of the area needs to be covered.
For an enclosed target area, the agents also needs to be covered within the closed belt.
For this reason, Binay et al. designed the algorithm to move the smart body to the boundary of a simple polygon and thus protect the area inside the polygon\cite{BB2009}.
Moreover, a barrier coverage algorithm has been designed on a circle, and the agent can be uniformly covered on the circle with limited communication\cite{song2018}.
But a circle is a kind of convex region, and how to perform barrier coverage on the boundary of non-convex regions is the inspiration of our research.
Moreover, covering only the boundary means that as soon as the intruder breaks through this layer, the intruder enters the area we need to protect and the system loses the means to monitor the intruder, which means an increase in security risks.
From a security point of view, a k-barrier coverage is more secure than a single layer barrier coverage.

To this end, we first designed algorithms that can perform barrier coverage on the boundary of a class of non-convex regions.
The algorithm is applied in the context of monitoring intruders, and uses a region partition to assign a region to each agent for monitoring, which can eventually lead to a local maximum monitoring probability.
Therefore, we want to design a multi-agent control algorithm with multi-layer barrier coverage to solve this problem. When an intruder breaks through a layer, there are still several internal monitoring layers that can continue to monitor the intruder.
The goal of this paper is to design a distributed multi-agent barrier coverage algorithm that can implement a multi-layer barrier coverage and can autonomously adjust the number of agents on each layer to optimize the monitoring quality of the whole system. The contributions of this paper are as follows.
\begin{enumerate}
\item	Design a multi-agent barrier coverage algorithm for a class of non-convex areas which can maximize intruder monitoring.
\item	Develop a distributed adjustment mechanism for the number of agents per layer, which can optimize the monitoring probability of multi-agent systems.
\item	Combining the single-layer fence coverage algorithm and the distributed adjustment mechanism of the number of multi-layer agents, we propose the multi-layer barrier coverage algorithm.
\end{enumerate}

The remainder of this paper is structured as follows: Section~\ref{sec:pro} presents a single-layer barrier coverage algorithm for non-convex region boundaries at first and then provides a distributed adjustment mechanism for the number of agents in a multi-layer coverage region and a multi-layer barrier coverage algorithm. Section~\ref{sec:mai} presents a theoretical verification of the single-layer barrier coverage algorithm and the multi-layer barrier coverage algorithm proposed in Section~\ref{sec:pro} and gives the case when our algorithm is applied to a circle. Section~\ref{sec:cas} simulates our algorithm and performs experimental validation on Robotarium. Finally, we conclude the paper in Section~\ref{sec:con}.

\section{Problem Formulation}\label{sec:pro}
In this section, we will introduce the distributed multi-agent barrier coverage algorithm.
Consider a closed curve region $D$, which can be represented by polar coordinates, the center of the circle $D$ is denoted by $O$.
Without loss of generality, we can set the center of the circle as the origin, i.e. $O=(0,0)$.
The boundary of the circle area $D$ is denoted by $\partial D$.
The radius of the closed curve region is denoted by $R(\theta),\theta\in [{0,2\pi })$.

\subsection{Single layer barrier coverage}
In the application of monitoring intruder, multiple layer barrier coverage is more effective than single layer barrier coverage. Therefore, we propose a multiple layer barrier coverage algorithm in this paper. Since this algorithm is based on single layer barrier coverage algorithm, we firstly introduce the single layer barrier coverage algorithm in this subsection.

In layer $k$, there are $N_k$ mobile agents that can communicate and monitor.
The layer $k$ can be denoted by $R_k(\theta), \theta \in [0,2\pi)$.
We assume that agents can all communicate with each other if they are on the same layer.
We use $I_{N_k}$ to denote the number of agents on this layer, $I_{N_k} = \{1,2,...,N_k\}$.
We use $\rho(\theta)$ to denote the probability that each point on the layer is invaded by an intruder.
We denote the position of agents by $P = \{ {{p_1},{p_2},...,{p_{N_k}}} \}$.
We specify an angle for agent $i_k$ with respect to the center of the circle $O$, denoted by $\varphi_{i_k}$, $i_k=1,2,...,N_k$.
We denote the probabilistic model that the agent $i_k$ detects an intruder by $f(d(\varphi_{i_k},\theta))$, where $d(\varphi_{i_k},\theta)$ is a distance function about $\varphi_{i_k}$ and $\theta$, and the distance function is Lipschitz continuous, and $d\left( {{\varphi _{i_k}},\theta } \right) \propto \delta(\varphi_{i_k},\theta)$. The function $\delta(\varphi_{i_k},\theta)$ is denoted by
\begin{equation}
        {\delta ({\varphi_{i_k},\theta})} = \left\{ {\begin{array}{*{20}{c}}
{ |\varphi_{i_k}-\theta- 2\pi| }&{if\quad\varphi_{i_k}-\theta > \pi }\\
{|\varphi_{i_k}-\theta + 2\pi| }&{if\quad\varphi_{i_k}-\theta \le  - \pi }\\
{|\varphi_{i_k}-\theta|}&{else}
\end{array}} \right.
\label{delta}
\end{equation}
Moreover, $f(d(\varphi_{i_k},\theta))$ need to meet the following conditions:
\begin{enumerate}
\item  $f(d(\varphi_{i_k},\theta))$ is differentiable.
\item $f(d(\varphi_{i_k},\theta))$ is monotonically decreasing.
\end{enumerate}
Now we give the calculation way of $\varphi_{i_k}$ as follows
\begin{equation}
    \varphi_{i_k}(t)=\Psi(p_{i_k}^{T}(t)).
    \label{varphit}
\end{equation}
where $\Psi(x,y)$ is an operation specified by us, which is calculated as follows
\begin{equation}
    \Psi ( {x,y} ) = \left\{ {\begin{array}{*{20}{c}}
  {\arctan ( {\frac{y}{x}} ) + \pi }&{x < 0}&{} \\
  {\arctan ( {\frac{y}{x}} )}&{y \geqslant 0}&{x > 0} \\
  {\arctan ( {\frac{y}{x}} ) + 2\pi }&{y < 0}&{x > 0} \\
  {\frac{\pi }{2}}&{y > 0}&{x = 0} \\
  { - \frac{\pi }{2}}&{y < 0}&{x = 0} \\
  0&{y = 0}&{x = 0}
\end{array}} \right.
    \label{phat}
\end{equation}
Combining (\ref{varphit}) and (\ref{phat}), it is easy to know that $\varphi_{i_k}\in[0,2\pi)$, for $i_k=1,2,...,N_k$. Moreover, agents have the following numbering rules
\[0 \le {\varphi _1}\left( 0 \right) < {\varphi _2}\left( 0 \right) < ... < {\varphi _{N_k}}\left( 0 \right) < 2\pi \]
\begin{figure}
{\includegraphics[width=10cm]{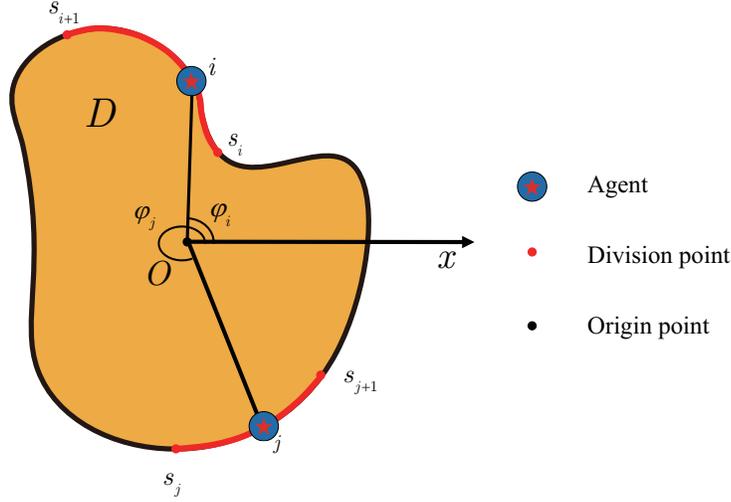}}\centering
\caption{\label{area} Coverage area $D$, agent communication radius and monitoring radius, and the color of the pentagram indicates the working state of the agent.}
\end{figure}

As the distance increases, the detection probability of the agent will decrease. Therefore, we stipulate that the agent only detects the area in which it is responsible.
Therefore, we propose a partition method that can partition the layer $k$ into $N_k$ subareas.
In order to achieve partition, we provide $N_k$ division points on the layer $k$, denoted by ${\rm S }= \{s_1,s_2,...,s_{N_k} \}$, and $s_{i_k} \in [0,2\pi)$ represent the phase of these division points.

These division points divide the layer $k$ into $N_k$ sub-areas, which is denoted by $E=\{E_1,E_2,...,E_{N_k}\}$, $E_{i_k}$ is represented as follows
\begin{equation}
    {E_{i_k}} = \left\{ {\begin{array}{*{20}{c}}
{\left\{ {(R_k(\theta),\theta) |{s _{i_k}} \le \theta < {s _{i_k + 1}}} \right\}}&{if\quad{s _{i_k}} < {s _{i_k+1}}}\\
{\left\{ {(R_k(\theta),\theta) |{s _{i_k}} \le \theta < 2\pi ,0 \le \theta < {s _{i_k + 1}}} \right\}}&{otherwise}
\end{array}} \right.
\label{Ei}
\end{equation}
and $s_{N+1}=s_1$. The probability of an intruder invading from $E_{i_k}$ is denoted by $m_{i_k}$, and $m_{i_k}$ is calculated as follows
\begin{equation}
{m_{i_k}} = \left\{ {\begin{array}{*{20}{c}}
{\int_{{s _{i_k}}}^{{s _{i_k + 1}}} {\rho \left( \theta  \right)d\theta } }&{if\quad{s _{i_k}} < {s _{i_k+1}}}\\
{\int_{{s _{i_k}}}^{2\pi } {\rho \left( \theta  \right)d\theta }  + \int_0^{{s _{i_k + 1}}} {\rho \left( \theta  \right)d\theta } }&{otherwise}
\end{array}} \right.
    \label{mi}
\end{equation}
In the same way, we use $\mathcal{T}_{i_k}$ to indicate where the division point exists as follows
\begin{equation}
    {{\rm{{\cal T}}}_{i_k}} = \left\{ {\begin{array}{*{20}{c}}
{\left\{ {(R_k(\theta),\theta)|{\varphi _{i_k}} \le \theta < {\varphi _{i_k+1}}} \right\}}&{if\quad {\varphi _{i_k}} < {\varphi _{i_k+1}}}\\
{\left\{ {(R_k(\theta),\theta)|{\varphi _{i_k}} \le \theta < 2\pi ,0 \le \theta < {\varphi _{i_k+1}}} \right\}}&{otherwise}
\end{array}} \right.
\label{regionT}
\end{equation}
where $\varphi_{N_k+1} = \varphi_1$.

According to the Law of Total Probability, we can give the monitoring probability of the multi-agent systems as follows
\begin{equation}
    H(\varphi,\mathcal{S}) = \sum\limits_{i = 1}^N {\int_{E_{i_k}} {f(d( \varphi _{i_k},\theta) )\rho (\theta )d\theta } }
    \label{H}
\end{equation}
It is not difficult to find that the meaning represented by the quality function $(\varphi,\mathcal{S})$ is the probability that an intruder intrusion is detected.
For applications that detect intruders, the larger the value of the equation (\ref{H}) the better the system. Thus the problem of detect intruders can be transformed into the following optimization problem
\begin{equation}\label{optimization}
  \max \left( {H\left( {\varphi ,\mathcal{S}} \right)} \right)
\end{equation}
In order to fit the actual situation, We express the agent dynamics equations with the following nonholonomic constraint motion equations
\begin{equation}
\left\{ {\begin{array}{*{20}{c}}
{{x_{i_k}} = {r_{i_k}}\cos \left( {{\varphi _{i_k}}} \right)}\\
{{y_{i_k}} = {r_{i_k}}\sin \left( {{\varphi _{i_k}}} \right)}\\
{{{\dot \varphi }_{i_k}} = \omega_{i_k} }\\
{{{\dot r}_{i_k}} = u_{i_k}^r}
\end{array}} \right.
    \label{XY}
\end{equation}
where $r_{i_k}$ represent the distance between the agent and the center of the circle, i.e. $r_{i_k}=\|p_{i_k}\|$.
$x_{i_k}$ and $y_{i_k}$ are the horizontal and vertical coordinates of $p_{i_k}$, and $p_{i_k}(t)= (x_{i_k}(t),y_{i_k}(t))^T$ represent the agent position at the time step $t\in\mathbb{R}^+$.
$\omega_{i_k}$ represents the angular velocity of the agent, which will be introduced later.
$u_{i_k}^r$ is denoted as follows
\begin{equation}
u_{i_k}^r={\kappa _r}\left( {{R(\varphi_{i_k} )} - {r_{i_k}}} \right),
    \label{uri}
\end{equation}
where $\kappa_r$ is an adjustable parameter.
From (\ref{XY}) and (\ref{uri}), we can get the dynamic equation of the agents is
\begin{equation}
\left\{ {\begin{array}{*{20}{c}}
{\dot{x}_{i_k} = \frac{{\partial x_{i_k}}}{{\partial {r_{i_k}}}}\dot{r_{i_k}} + \frac{{\partial x_{i_k}}}{{\partial {\varphi _{i_k}}}}\dot{\varphi _{i_k}} = u_{i_k}^r\cos \left( {{\varphi _{i_k}}} \right) - {r_{i_k}}\omega_{i_k} \sin \left( {{\varphi _{i_k}}} \right)}\\
{\dot{y}_{i_k} = \frac{{\partial y_{i_k}}}{{\partial {r_{i_k}}}}\dot{r_{i_k}} + \frac{{\partial y_{i_k}}}{{\partial {\varphi _{i_k}}}}\dot{\varphi _{i_k}} = u_{i_k}^r\sin \left( {{\varphi _{i_k}}} \right) + {r_{i_k}}\omega_{i_k} \cos \left( {{\varphi _{i_k}}} \right)}
\end{array}} \right.
    \label{dyn_p}
\end{equation}
Using the gradient method for (\ref{H}), we can get
\begin{equation}
\begin{aligned}
\dot H &= \sum\limits_{i = 1}^N {\int_{{E_{i_k}}} {\frac{{\partial f\left( {d({\varphi _{i_k}},\theta )} \right)}}{{\partial {\varphi _{i_k}}}}\rho \left( \theta  \right)d\theta } }  \cdot {\omega _{i_k}} + \sum\limits_{i = 1}^N {\frac{{\partial H}}{{\partial {s _{i_k}}}}{{\dot s }_{i_k}}}
\end{aligned}
\label{dh}
\end{equation}
To maximize $H$, we can set $\omega_{i_k}$ as follows
\begin{equation}
{\omega}_{i_k}=\kappa_\omega  \int_{{E_{i_k}}}  {\frac{{\partial f\left( {d({{\varphi _{i_k}},\theta })} \right)}}{{\partial {\varphi _{i_k}}}}\rho \left( \theta  \right)d\theta }
    \label{omega}
\end{equation}
where $\kappa_\omega $ is a adjustable constant ,and (\ref{omega}) can guarantee that (\ref{dh}) is not less than zero, which means that $H$ does not decrease.

We construct the following control input of division points
\begin{equation}
\dot{s}_{i_k}=\kappa_{s}(d(\varphi_{i_k},s_{i_k})-d(\varphi_{i_k-1},s_{i_k}))
\label{coni_s}
  \end{equation}
where
$\kappa_s$
is positive constant.
In order to apply the algorithm to the multi-layer barrier coverage algorithm, we need to make some adjustments to the control input.
We can find that for agent $i_k$, which performs barrier coverage, only the states of agent $i_k -1$ and agent  $i_k +1$ are needed to complete the algorithm. Moreover, the relative position of agent $i_k -1$ and agent $i_k +1$ is the closest agent in the clockwise and counterclockwise direction of agent $i_k$, respectively.
Therefore, in multi-layer coverage problem, we use $\alpha$ and $\beta$ to denote agent $i_k-1$ and agent $i_k+1$. We can rewrite the control input of the agent as follows
\begin{equation}
\begin{aligned}
{{\dot s}_{{i_k}}} &= {\kappa _s}(d({\varphi _{{i_k}}},{s_{{i_k}}}) - d({\varphi _\alpha },{s_{{i_k}}}))\\
{\omega _{{i_k}}} &=
\left\{ {\begin{array}{*{20}{c}}
{{\kappa _\omega }\int_{{s_{{i_k}}}}^{{s_\beta }} {\frac{{\partial f\left( {d({\varphi _{{i_k}}},\theta )} \right)}}{{\partial {\varphi _{{i_k}}}}}\rho \left( \theta  \right)d\theta } },&{if\:\:{\mkern 1mu} {s_{{i_k}}} < {s_\beta }}\\
{{\kappa _\omega }\int_{{s_{{i_k}}}}^{2\pi } {\frac{{\partial f\left( {d({\varphi _{{i_k}}},\theta )} \right)}}{{\partial {\varphi _{{i_k}}}}}\rho \left( \theta  \right)d\theta }  + \int_0^{{s_\beta }} {\frac{{\partial f\left( {d({\varphi _{{i_k}}},\theta )} \right)}}{{\partial {\varphi _{{i_k}}}}}\rho \left( \theta  \right)d\theta } }.&{otherwise}
\end{array}} \right.
\end{aligned}
    \label{Single_layer_con}
\end{equation}

Finally, we give the single layer barrier distributed coverage algorithm as in Table.\ref{tab:single_layer}. In this we ensure that the split point must lie at the midpoint of the curve between the intelligences. Since the splitting point is virtual, this step is quite fast in practical execution.
The multi-layer barrier coverage algorithm is described next.
\begin{algorithm}[t]
\caption{\label{tab:single_layer} Single Layer Barrier Coverage Algorithm}
\hspace*{0.02in}
{\bf Initializate: $\kappa_r$, $\kappa_\omega $, $\kappa_ s$, $T^*$}\\
 For $i_k \in I_{{N}_k}$, $i_k$-th agent performs as follow
\begin{algorithmic}[1]
\For{$t=1:T^*$}
       \State Calculate $\varphi_{i_k}$ by (\ref{varphit}) and (\ref{phat});
       \While{$d({\varphi _{{i_k}}},{s_{{i_k}}}) - d({\varphi _\alpha },{s_{{i_k}}})<\varepsilon$}
       \State Update $s_{i_k}$ with (\ref{Single_layer_con});
       \EndWhile\State {\bf end~while}
       \State Update $p_{i_k}$  with (\ref{Single_layer_con}), (\ref{uri}) and (\ref{XY});
\EndFor \State {\bf end~for}
\end{algorithmic}
\end{algorithm}

\subsection{Multi-layer barrier coverage}
In this subsection, we will introduce a distributed barrier coverage control algorithm based on subsection 2.1.

Consider $K^*$ layers of area to be covered. We use $R_1(\theta), R_2(\theta),...,R_{K^*}(\theta)$ to denote the polar coordinate equation of these layers, and $0<R_1(\theta)< R_2(\theta)<...<R_{K^*}(\theta)<R_{max}$, for $\theta \in (0,2\pi]$. Where $R_{max}$ is a positive constant. From subsection 2.1, the number of agents on layer $k$ is denoted by ${N_k}$. In the same way, the number of all agents on the layer is denoted by $N_L$, and $N_L =\sum_{k=1}^{K^*} N_k.$

Rather than the number of agents in each layer being fixed, we prefer to find a distributed method that can automatically allocate the number of agents in each layer.
We call this function as layer swapping. Agent will get a target layer when it is going to do layer swapping.
It is easy to find when an agent moves to its target layer, the agent does not belong to any layer. We call this class of agents as free agent. On the other hand, agents belong to a layer are called as layer agent.
Moreover, we think that there should be no difference between the states of agents at the initial moment except their distinct positions. Therefore, all the agents are free agent at the initial moment in our work. We can think of free agent as stem cell and layer agent as differentiated cell.
The transformation of free agent into layer agent is like the differentiation of cell.
The number of free agent is denoted by $N_F$. The number of all agents is represented by $N$, and $N= N_L+N_F$. In the initial moment, $N=N_F$. Here we numbered all the agents as $I_{N}=\{1,2,...,N\}$. We use $a_i$ to denote what type of agent is the agent $i$, when $a_i=0$, the agent is a free agent and when $a_i =1$, it is a layer agent. And we use the Algorithm \ref{tab:number} to calculate $N_k$ for each agent, which is the basis of our work.

\begin{algorithm}[t]
\caption{\label{tab:number} $N_k$ Calculate Algorithm}
\hspace*{0.02in}
\begin{algorithmic}[1]
\For{$j=1:N$}
    \If{$ a_j = 1$}
        \State{$k = k_j$;}
        \State{$N_k = N_k+1$;}
        \State{$I_{N_k}=I_{N_k}\cup\{j\}$;}
    \EndIf \State{\bf end~if}
\EndFor \State{\bf end~for}
\end{algorithmic}
\end{algorithm}

It is similar to that shown in subsection 2.1, $\varphi_i$ and $p_i= (x_i,y_i)^T$ denote the phase angle and position of agent $i$, respectively.
And we use $k_i$ to denote the target layer of agent $i$, and $k_i \in \{0,1,2,...,K\}$. $k_i=0$ means agent $i$ has no target layer. In this case, in order to find target layer agent $i$ will move as follows
\begin{equation}
\left\{ \begin{aligned}
\dot{x}_i &=  - {r_i}{\omega _0}\sin \left( {{\varphi _i}} \right),\\
\dot{y}_i &= {r_i}{\omega _0}\cos \left( {{\varphi _i}} \right),
\end{aligned}
\right.
    \label{u_no_target}
\end{equation}
where $\omega _0$ is a constant. And when $k_i  \ne 0$, similar to control input (\ref{u_no_target}) , agent will move as follows
\begin{equation}
  \left\{ \begin{aligned}
        \dot{x}_i &=  \kappa_r(R(\varphi_i)-r_i)cos(\varphi_i), \\
        \dot{y}_i &= \kappa_r(R(\varphi_i)-r_i)sin(\varphi_i).
   \end{aligned}
     \right.
\label{u_get_target}
\end{equation}

In subsection 2.1, we have numbered the agent. However, there are some difference about agent number in this subsection. In layer $k$, $I_{N_k} = \{j|a_j\times k_j=1\}$. As we calculated in Algorithm \ref{tab:number}, we use $I_{N_k}$ to denote the number set of layer $k$. If all agents work on layers, there is such a relationship that $I_N = \bigcup_{k=1}^{K^*}I_{N_k} $.

When the agent is close to a certain layer, the agent needs to consider whether it c an join the covering task of this layer.
We use $r_k$ to denote the range of layer $k$, and $r_k :=\{(rcos(\theta),rsin(\theta))|R_k(\theta)-\Delta \le r\le R_k(\theta)+\Delta\}$. And $\Delta$ is a is a small enough constant.
When an agent enters $r_k$, we consider that the agent is close to layer $k$.
Moreover ,we consider that whether agent $i$ can enter the $k$-th layer depends on agent $j$ already working in the $k$-th layer, rather than agent $i$ itself.
And only when agent $j$ approves this entry, agent $i$ can enter layer $k$ to perform the detection task, otherwise agent $i$ should try to move to other layers.
If there is no agent in layer $k$, the agent will enter the layer $k$ without any problem.

Now, we will introduce the detect state of agent $i$, when agent $i$ is performing the detect task.
The detect state of agent $i$ is the key variable to judge whether the agent outside the layer can enter the layer to execute the task.
It is easy to know that when an agent is carrying too much work, its detection capability will decrease.
On the other hand, when there are enough agents in a certain layer, the contribution of agents entering this layer is not as large as that entering other layers.
Therefore, we use $c_{i}$ to denote the detect state of agent $i$ as follows
\begin{equation}
{c_{i}} = \left\{ {\begin{array}{*{20}{c}}
1&{\eta_{i}  > h}\\
0&{otherwise}
\end{array}} \right.
\label{ci}
\end{equation}
where $h\in(0,1)$ is an adjustable parameter, and
\begin{equation}
\eta_{i}  = \frac{{\int_{E_i}{f(d({\varphi _{{i}}},\theta ))\rho (\theta )d\theta }}}  {{\int_{{E_{{i}}}} {\rho (\theta )d\theta } }},
\label{eta}
\end{equation}
$\eta_{i}$ can be interpreted as the task completion rate. It can also be interpreted as the probability of being detected by agent $i$ under the condition that intruder invades region $E_{i}$.

\begin{figure}
{\includegraphics[width=0.85\linewidth]{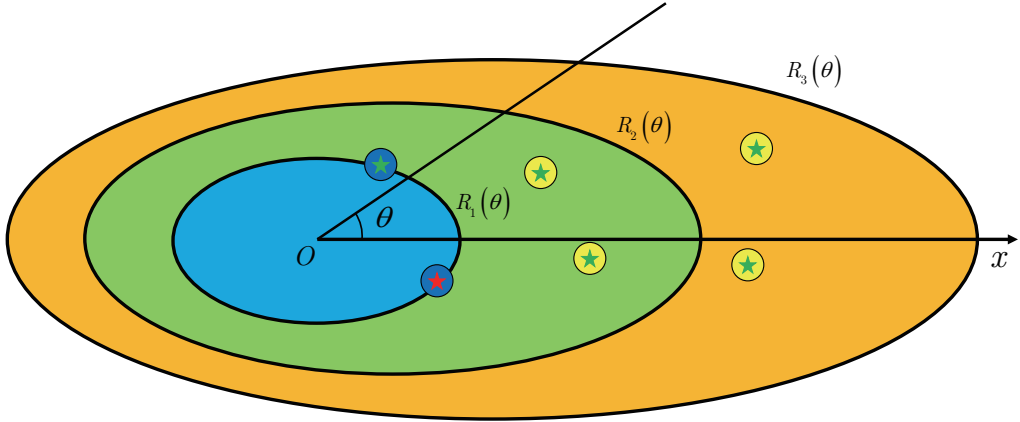}}\centering
\caption{\label{f_multilayers_area} Coverage area $D$, and agent communication radius and monitoring radius.}
\end{figure}
As shown in Fig.\ref{f_multilayers_area}, there are three layers of area to cover. The color of the star represents the detection state of the agent, and the color of the circle represents whether the agent is a free agent.
Blue circle means this agent is performing detect task, and yellow circle means this agent is a free agent and moving to its target layer.
Red star means that this agent does not allow other agents to enter this layer. Green star means that this agent allows other agent enter this layer. And the star will turn red if $c_{i}=1$.

Every free agent has a target layer, we use $k_i$ to denote the target layer of agent $i$, and how to help the agent find the target layer is the basis of the algorithm.
We present Algorithm \ref{tab:target_layer_find} to help the agent achieve this function. Our idea is that all agents target the first layer first. If you cannot enter to the first layer, consider the second, and so on, all the way to the $K^*$-th layer. When considering whether to take the $k$-th  layer as the target layer, if there is no agent at the $k$-th layer, agent $i$ will choose to target at the $k$-th layer. If there are agents in $k$-th layer, agent $i$ needs to predict whether the agent at layer k can allow entry.
\begin{algorithm}[t]
 \caption{\label{tab:target_layer_find} Target Layer Identification Algorithm}
\hspace*{0.02in}
Agent $i$ performs as follows
\begin{algorithmic}[1]
\For{$k = 1:K^*$}
        \If{$N_k=0$}
            \State{$k_i = k$;}
        \Else
            \For{$j\in I_{N_k}$}
                \If{$(\varphi_i,R_k(\varphi_i))\in E_{j}$}
                    \If{$c_j=0$}
                        \State{$k_i = k_j$;}
                    \EndIf\State{\bf end~if}
                \EndIf \State{\bf end~if}
            \EndFor\State{\bf end~for}
        \EndIf \State{\bf end~if}
\EndFor\State{\bf end~for}
\State{\bf return} $k_i$
\end{algorithmic}
\end{algorithm}

Now, we need to consider how to let an agent enter a layer. As mentioned above, whether an agent can enter the layer depends on the agent in the layer. Therefore, we use the Algorithm \ref{tab:request} to realize this function. In Algorithm \ref{tab:request}, in order to prevent the phase of the agent from being the same, resulting in the difficulty of setting subsequent division points, the agent will change its phase when it finds the phase is the same. To avoid agents being preempted by other agents when they change phase, we need to set $a_i=1$ first.
\begin{algorithm}[t]
 \caption{\label{tab:request} Entry Request Algorithm}
\begin{algorithmic}[1]
\If{$N_k=0$}
    \State{$a_i = 1$;}
     \State $\alpha=\beta=i$;
    \State $s_i = \varphi_i +\pi$;
    \If{$s_i>2\pi$}
        \State $s_i = s_i - 2\pi$;
    \EndIf\State{\bf end~if}
\Else
    \For{$j \in I_{N_k}$}
        \If{$(\varphi_i,R_k(\varphi_i))\in E_{j}$}
            \If{$c_{j}=1$}
            \State{$k_i = 0$;}
            \Else
            \State{$a_i = 1$;}
            \While{$\varphi_i = \varphi_{j}$}
                \State{Move with (\ref{u_no_target});}
                \State {Calculate $\varphi_{i}$ by (\ref{varphit}) and (\ref{phat});}
            \EndWhile \State{\bf end~while}
            \EndIf \State{\bf end~if}
        \EndIf \State{\bf end~if}
    \EndFor \State{\bf end~for}
\EndIf\State{\bf end~if}
\end{algorithmic}
\end{algorithm}

It is easy to find that the Algorithm \ref{tab:single_layer} requires agent $\alpha$ and agent $\beta$ to implement.
Therefore, we design the Algorithm \ref{tab:neighbor_find} to find the agent $\alpha$ and agent $\beta$ for agent $i$. In Algorithm \ref{tab:neighbor_find}, we design a operation similar to (\ref{delta}) as follows
\begin{equation}
        {\delta^* ({\theta_1,\theta_2})} = \left\{ {\begin{array}{*{20}{c}}
{ \theta_1-\theta_2- 2\pi }&{if\quad\theta_1-\theta_2 > \pi }\\
{\theta_1-\theta_2 + 2\pi }&{if\quad\theta_1-\theta_2 \le  - \pi }\\
{\theta_1-\theta_2}&{else}
\end{array}} \right.
    \label{delta_star}
\end{equation}
Actually, $\delta(\theta_1,\theta_2)=|\delta^*(\theta_1,\theta_2)|$.
The operation can calculate the phase difference from $\theta_1$ to $\theta_2$ in the counterclockwise direction.
We can find that $\delta^*(\theta_1,\theta_2)\in (-\pi,\pi]$, which is difficult to to compare in algorithm.
Therefore, we use $\Psi(cos(\delta^*(\theta_1,\theta_2)),sin(\delta^*(\theta_1,\theta_2)))$ to let all the phase differences be positive.
Then, by finding the minimum of these, the agent $\alpha$ in the counterclockwise direction can be determined. In the same way, we can also get the agent $\beta$ in the other direction.

\begin{algorithm}[t]
\caption{\label{tab:neighbor_find} Neighbor Seeking Algorithm}
\hspace*{0.02in}
{\bf Input:} $I_{N_k}$, agent $i,j$ state\\
\hspace*{0.02in}
{\bf Output:} $\alpha$, $\beta$
\begin{algorithmic}[1]
        \For{$j\in I_{N_k}\&\&~j\ne i$}
        \If{$N_k = 2$}
            \State $\alpha=\beta=j$;
        \Else
        \If{$\Psi(cos(\delta^*(\varphi_{i},\varphi_{j})),sin(\delta^*(\varphi_{i},\varphi_{j})))<\Psi(cos(\delta^*(\varphi_{i},\varphi_{\alpha})),sin(\delta^*(\varphi_{i},\varphi_{\alpha})))$}
                \State $\alpha=j$;
        \EndIf \State{\bf end if}
        \If{$\Psi(cos(\delta^*(\varphi_{j},\varphi_{i})),sin(\delta^*(\varphi_{j},\varphi_{i})))<\Psi(cos(\delta^*(\varphi_{\beta},\varphi_{i})),sin(\delta^*(\varphi_{\beta},\varphi_{i})))$}
                \State $\beta = j$;
        \EndIf \State{\bf end if}
        \EndIf \State{\bf end if}
        \EndFor \State{\bf end~for}
        \State {\bf Return} $\alpha$ and $\beta$;
\end{algorithmic}
\end{algorithm}

When the neighbor agent $\alpha$ in clockwise direction changes, the division point $s_i$ should also change accordingly. The same is true for counterclockwise which does not change the division point $s_i$. When agent $\alpha$ does not change, agent $i$ can execute Algorithm \ref{tab:single_layer}. This ensures that Algorithm \ref{tab:single_layer} works efficiently.
\begin{algorithm}[t]
 \caption{\label{tab:number_and_division} Division Point Set Algorithm}
 \hspace*{0.02in} {\bf Initializate:} $z = \alpha$
\begin{algorithmic}[1]
\State Run Algorithm \ref{tab:neighbor_find};
\If{$z \ne \alpha$}
    \State{The division point is set as $s_{i}=\varphi_i-\frac{1}{2}\delta(\varphi_i,\varphi_\alpha)$;}
    \If{ $s_{i}<0$}
        \State $s_{i}=s_{i}+2\pi$;
    \EndIf \State{\bf end~if}
\Else
    \State Run Algorithm \ref{tab:single_layer};
\EndIf \State{\bf end~if}
\end{algorithmic}
\end{algorithm}

\begin{figure}[t!]
{\includegraphics[width=0.85\linewidth]{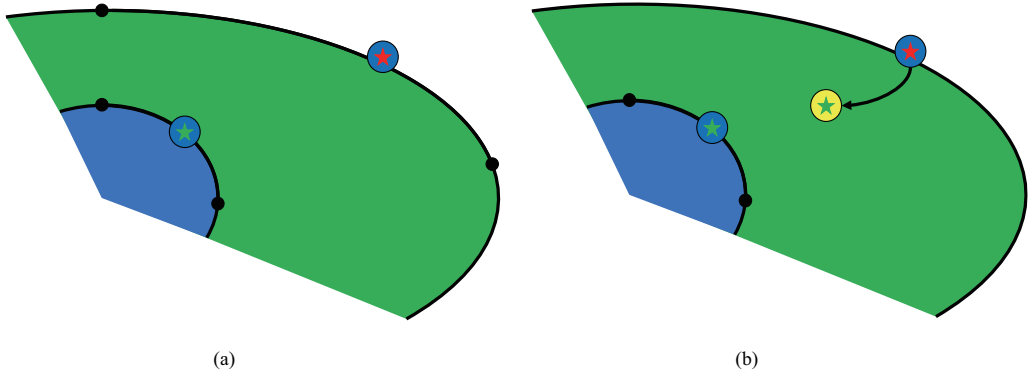}}\centering
\caption{\label{f_layer_change} Diagram of agent moving to other layer}
\end{figure}
In addition, the importance of each layer should be different from one another. In general, the more important the inner layer is.
Therefore, when the number of agents is insufficient, the inner layer should be covered first.
As shown in Fig.\ref{f_layer_change}, when the outer agent finds that the inner agent needs help, even if it is working well, it will leave the outer layer and head to the inner layer.
We use Algorithm \ref{tab:layer_change} to realize this function. When agent $i$ discovers $\varphi_i\in E_{j}$, $a_j = 1$ and $c_j = 0$ of the inner agent, the agent $i$ transforms itself into a free agent, and the inner layer is the target layer.

\begin{algorithm}[t]
\caption{\label{tab:layer_change} Weight-based Layer Change Algorithm}
\begin{algorithmic}[1]
    \If{$k_i>1$}
        \State Run Algorithm \ref{tab:number};
        \For{$j\in I_{N_{k_i-1}}$}
            \If{$\varphi_i\in E_{j}~\&\& ~a_j = 1~ \&\&~ c_j = 0$}
                    \State{Set $k_i = k_j$ and $a_i = 0$};
            \EndIf \State{\bf end if}
        \EndFor \State{\bf end~for}
    \EndIf\State{\bf end if}
\end{algorithmic}
\end{algorithm}

Finally, we present the multi-layer barrier coverage algorithm in Algorithm \ref{tab:multi_layer_coverage}.
When agent $i$ does not have a target layer, the agent will first find a target layer.
If agent $i$ cannot find the target layer with the phase unchanged, the agent will change its phase. After the agent finds the target layer, the agent moves to the target layer.
When the agent reaches the target layer, it will request to enter the target layer. If the request is rejected, the agent looks for another target layer.
When the request is granted, the agent enters the layer to perform the coverage task.
In order to ensure the smooth progress of the algorithm, the intelligent experience obtains the neighbor information at all times.
Finally, when the agent finds that the inner layer needs help, it stops coverage and helps the inner layer instead.
We use $P$ to denote the detection probability of the algorithm.
$P$ is calculated as follows
\begin{equation}\label{eq:multi_layer_P}
P = 1-\prod\limits_{k=1}^{N_k} (1-P_k),
\end{equation}
where $P_k$ is the detected probability of layer $k$.
\begin{algorithm}[t]
 \caption{\label{tab:multi_layer_coverage} Multi-layer Barrier Coverage Algorithm}
\hspace*{0.02in} {\bf Initializate:} $k=1$, $a_i=0$, $c_i=0$ \\
 For $i \in I_N$, $i$-th agent performs as follow
\begin{algorithmic}[1]
\While{$a_i=0$}
    \While{$k_i=0$ }
        \State{Run Algorithm \ref{tab:number};}
        \State{Run Algorithm \ref{tab:target_layer_find};}
        \State{Move with (\ref{u_no_target});}
    \EndWhile \State{\bf end~while}
    \While{$p_i\notin r_{k_i}$}
        \State{Move to the target layer with (\ref{u_get_target});}
    \EndWhile \State{\bf end~while}
    \State{Run Algorithm \ref{tab:number};}
    \State{Run Algorithm \ref{tab:request};}
    \While{$a_i = 1$}
        \State{Run Algorithm \ref{tab:number};}
        \State{Run Algorithm \ref{tab:number_and_division};}
        \State{Run Algorithm \ref{tab:layer_change};}
        \State Update $c_i$ with (\ref{ci}) and (\ref{eta});
    \EndWhile\State{\bf end~while}
\EndWhile \State{\bf end~while}
\end{algorithmic}
\end{algorithm}

In the next section, we will theoretically demonstrate the effectiveness of the proposed algorithm.

\section{Main Results}\label{sec:mai}
\begin{lemma}
For fixed agents position, the set of midpoints of $\mathcal{T}$ guarantees the maximum of joint monitoring probability $H$.
\label{le:division_point}
\end{lemma}
\begin{proof}
By taking the partial derivative of $H$ with respect to $s_i$, one gets
\begin{equation*}
    \frac{{\partial H}}{{\partial s_i}} = \left[ {f\left( {d({\varphi _{i - 1}},{s _i})} \right) - f\left( {d({\varphi _i},{s _i})} \right)} \right]\rho \left( {{s _i}} \right).
\end{equation*}
It is observed that if $f({d({\varphi _{i - 1}},{s _i})}) =f(d({\varphi _i},{s _i}))$ for $i= 1,2,...,N$, we can get $ \frac{{\partial H}}{{\partial \mathcal{S}}}=0$. According to the description of the properties of $f (\cdot)$ in Section II, this means that $ \frac{{\partial H}}{{\partial \mathcal{S}}}=0$ can be achieved with only $d({\varphi _{i - 1}},{s _i}) =d({\varphi _i},{s _i})$ for $i= 1,2,...,N$.
Since the distinct agents position, $d({\varphi _{i - 1}},{s _i}) =d({\varphi _i},{s _i})$ means division point is the midpoint of $\mathcal{T}_i$. Moreover, the Hessian matrix of the function of coverage quality (\ref{H}) satisfies
\begin{equation*}
\begin{split}
        {\nabla ^2}H &= [\frac{\partial^2U}{\partial s_i\partial s_j}]\in R^{N\times N}\\
        &=\text{diag}(\alpha_1, \alpha_2, ... ,\alpha_N ),
\end{split}
\end{equation*}
where $\alpha_i = (\frac{\partial f(d(\varphi_{i-1},s_i))}{\partial s_i}- \frac{\partial f(d(\varphi_{i},s_i))}{\partial s_i})\rho(s_i)$.
Since $f(\cdot)$ is monotonically decreasing and $d\left( {{\varphi _i},s_i } \right) \propto \delta(\varphi_i,s_i)$, combining with equation (\ref{delta}), we can get $ \frac{\partial f(d(\varphi_{i-1},s_i))}{\partial s_i}<0$ and $\frac{\partial f(d(\varphi_{i},s_i))}{\partial s_i}\rho(s_i)>0$.
This means $\alpha_i<0$ for $i=1,2,...,N$. Therefore, we can get ${\nabla ^2}H<0$, which implies this lemma.
\end{proof}

\begin{theorem}
Dynamic system (\ref{XY}) and (\ref{coni_s}) ensure that the function (\ref{H}) reach the local maximum value.
\label{Theorem:H}
\end{theorem}
\begin{proof}
Construct the following Lyapunov function
\begin{equation*}
V(\varphi,\mathcal{S}) = \frac{1}{H}.
\end{equation*}
Since $s_i \in [0,2\pi)$, $f(d(\varphi_i , \theta))>0$ and $\rho(\theta) \ge 0$, we can find that $V>0$. Moreover, $f(d(\varphi_i , \theta))$ and  $\rho(\theta)$ are bounded, which implies $V_1$ is bounded.
Taking the derivative of the Lyapunov function, we find that
\begin{equation*}
\dot V(\varphi,\mathcal{S}) = - \frac{1}{H^2} \cdot \dot H,
\end{equation*}
From (\ref{H}), we can find that $V(\varphi,\mathcal{S})>0$.The time derivative of (\ref{H}) with respect to the compound dynamics (\ref{XY}) and (\ref{coni_s}) is given by
\begin{equation*}
\begin{aligned}
    \dot H  &= \sum\limits_{i = 1}^N {\frac{{\partial H}}{{\partial {\varphi _i}}}{\omega _i}}  + \sum\limits_{i = 1}^N {\frac{{\partial H}}{{\partial {s _i}}}{{\dot s }_i}} \\
  & = \sum\limits_{i = 1}^N {{{\left( {\int_{{E_i}} {\frac{{\partial f\left( {d({\varphi _i},\theta )} \right)}}{{\partial {\varphi _i}}}\rho \left( \theta  \right)d\theta } } \right)}^2}} \\
  &+ {\kappa _s}\sum\limits_{i = 1}^N {\left[ {f\left( {d({\varphi _{i - 1}},{s_i})} \right) - f\left( {d({\varphi _i},{s_i})} \right)} \right]} \left( {d({\varphi _i},{s_i}) - d({\varphi _{i - 1}},{s_i})} \right)\rho \left( {{s_i}} \right)
\end{aligned}
\end{equation*}
Since $\left[ {f\left( {d({\varphi _{i - 1}},{s_i})} \right) - f\left( {d({\varphi _i},{s_i})} \right)} \right] \left( {d({\varphi _i},{s_i}) - d({\varphi _{i - 1}},{s_i})} \right) \ge 0$,
 we can find that $\frac{dH}{dt} \ge 0$. On the other hand, the derivative of Lyapuonv function satisfies $\dot V_1 \le 0$.
According to local invariant set theorem, the state of the system will converge to the set of $\{(\varphi,s) |\dot V_1 =0\}$.
From the equation (\ref{H}), we can know that $H$ is bounded. Therefore, if and only if $\frac{dH}{dt}=0$, $\dot{V_1}=0$.
And in this case, the division points are located in the middle of the arc lengths between agents. In the meantime, agents are located in the set that $\{\varphi_i|\int_{E_i}\frac{\partial f(d(\varphi_i,\theta))}{\varphi} \rho(\theta)d\theta=0 \}$. Therefore, $V$ reach the local minimum value. On the other hand, $H$ reach the local maximum value.
\end{proof}

\begin{lemma}
For a working agent $i$, the agent $i$ will never collide with the division point.
\label{lemma_agent_point_collide}
\end{lemma}
\begin{proof}
We assume that the agent $i$ enters the layer at time $t^*$, we can get the following relationship
\begin{equation*}
    \delta^*(s_\beta,\varphi_i)>0, \delta^*(\varphi_i,s_\alpha)>0
\end{equation*}
Let us first consider that agent $i$ will not collide with division point $s_\beta$.
As shown in Fig.\ref{f_lemma_coliside}, we use $E_i^1,E_i^2,E_i^3,E_i^4$ to denote the area between two phases.
\begin{figure}
{\includegraphics[width=0.85\linewidth]{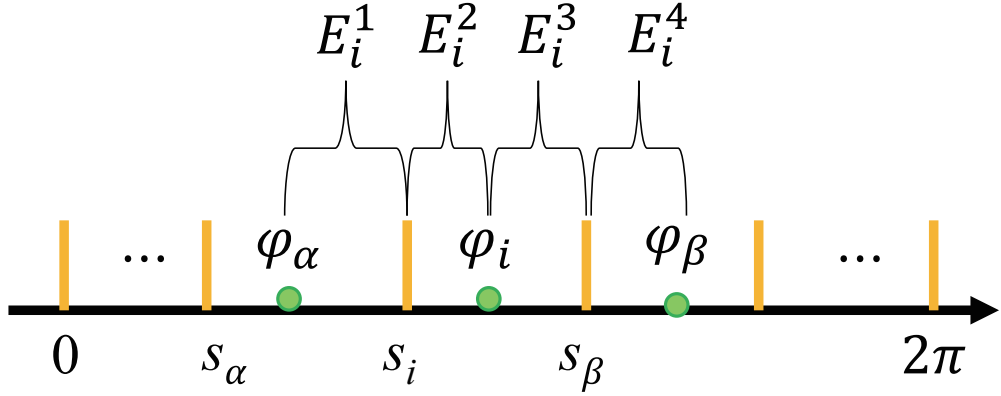}}\centering
\caption{\label{f_lemma_coliside} Diagram of the phase location of agent $i$}
\end{figure}
We use $L_\beta$ to denote the distance between agent $i$ and division point $s_\beta$, and $L_\beta$ is represented as follows
\begin{equation*}
    L_\beta = \mathcal{K}(\delta^*(s_\beta,\varphi_i)),
\end{equation*}
where $\mathcal{K}(\cdot)$ is a class $\mathcal{K}$ function.
Next, in combination with Equation (\ref{delta_star}), we take the derivative of $L_\beta$ to get
\begin{equation*}
\begin{aligned}
        \dot{L}_{\beta} & = \mathcal{K}_1\cdot(\dot{s_\beta}-\dot{\varphi_i})\\
        &=\mathcal{K}_1\cdot(\kappa_s(\delta^*(\varphi_\beta,s_\beta)-\delta^*(s_\beta,\varphi_i))-\int_{E_i}{\frac{\partial f(d(\varphi_i,\theta))}{\partial\varphi_i}\rho(\theta)d\theta})\\
        & = \mathcal{K}_1\cdot(\kappa_s(\delta^*(\varphi_\beta,s_\beta)-\delta^*(s_\beta,\varphi_i))-\int_{E_i^3}{\frac{\partial f(d(\varphi_i,\theta))}{\partial\varphi_i}\rho(\theta)d\theta}-\int_{E_i^2}{\frac{\partial f(d(\varphi_i,\theta))}{\partial\varphi_i}\rho(\theta)d\theta})
\end{aligned}
\end{equation*}
where $\mathcal{K}_1 = \frac{\partial \mathcal{K}(\delta^*(s_\beta,\varphi_i))}{\partial \delta^*(s_\beta,\varphi_i)}>0$.
As $\varphi_i\longrightarrow s_\beta$,  we can find that $\int_{E_i^3}{\frac{\partial f(d(\varphi_i,\theta))}{\partial\varphi_i}\rho(\theta)d\theta}\longrightarrow 0$, $\frac{\partial f(d(\varphi_i,\theta))}{\partial\varphi_i}<0$ in the range of $E_i^2$, and $\delta^*(\varphi_\beta,s_\beta)>0$. Therefore, we can get that
\begin{equation*}
    \dot L_\beta >-\mathcal{K}_1 \cdot \kappa_s(\delta^*(s_\beta,\varphi_i)),
\end{equation*}
which means that $L_\beta >0$ holds within the interval of $t\ge t^*$. In the same way, we use the $L_\alpha$ to denote the distance between agent $i$ and division point $s_i$, i.e.$L_\alpha=\mathcal{K}(\delta^*({\varphi_i},{s_i}))$, and we can also get the following
\begin{equation*}
\begin{aligned}
            \dot{L_i} &= \mathcal{K}_2\cdot(\dot{\varphi_i}-\dot{s_i})\\
            &=  \mathcal{K}_2\cdot(\int_{E_i}{\frac{\partial f(d(\varphi_i,\theta))}{\partial\varphi_i}\rho(\theta)d\theta}-\kappa_s(\delta^*(\varphi_i,s_i)-\delta^*(s_i,\varphi_\alpha)))\\
            &=  \mathcal{K}_2\cdot(\int_{E_i}{\frac{\partial f(d(\varphi_i,\theta))}{\partial\varphi_i}\rho(\theta)d\theta}-\kappa_s(\delta^*(\varphi_i,s_i)-\delta^*(s_i,\varphi_\alpha)))\\
            &=\mathcal{K}_2\cdot(\int_{E_i^2}{\frac{\partial f(d(\varphi_i,\theta))}{\partial\varphi_i}\rho(\theta)d\theta}+\int_{E_i^3}{\frac{\partial f(d(\varphi_i,\theta))}{\partial\varphi_i}\rho(\theta)d\theta}+\kappa_s\delta^*(s_i,\varphi_\alpha)-\kappa_s\delta^*(\varphi_i,s_i))
\end{aligned}
\end{equation*}
where $\mathcal{K}_2 = \frac{\partial \mathcal{K}(\delta^*(s_i,\varphi_i))}{\partial \delta^*(s_i,\varphi_i)}>0$.
As $\varphi_i\longrightarrow s_i$,  we can find that $\int_{E_i^2}{\frac{\partial f(d(\varphi_i,\theta))}{\partial\varphi_i}\rho(\theta)d\theta}\longrightarrow 0$, $\frac{\partial f(d(\varphi_i,\theta))}{\partial\varphi_i}>0$ in the range of $E_i^3$, and $\delta^*(s_i,\varphi_\alpha)>0$. Therefore, we can get that
\begin{equation*}
    \dot L_\alpha >-\mathcal{K}_2\cdot\kappa_s (\delta^*(\varphi_i,s_i)),
\end{equation*}
which means that $L_\alpha >0$ holds within the interval of $t\ge t^*$. Because of $L_\alpha>0$ and $L_\beta>0$, the agent will not collide with the division point.
\end{proof}

\begin{lemma}
The division points never collide with each other.
\label{le:divison_collide}
\end{lemma}
\begin{proof}
From lemma \ref{lemma_agent_point_collide}, we can know that the distance between division point $s_i$ and $s_\beta$ can be denoted as $L_i=L_\alpha+L_\beta$. Obviously, $L_i$ is the length of $E_i$. Therefore, we can get that $L_i>0$ holds within the interval of $t\ge t^*$, which implies this lemma.
\end{proof}

We use $L^k$ to denote the length of layer $k$. To demonstrate our conclusion, we discover the following lemmas.
\begin{lemma}
For agent $i$ working at layer $k$, if $L_i= min\{j\in I_{N_k}|L_j\}$, we have $L_i\le \frac{L^k}{N_k}$.
\label{le:min_l}
\end{lemma}
\begin{proof}
From Table \ref{tab:number_and_division}, The region of the $k$-th layer will be divided without remainder by the agents working in the $k$-th layer. Therefore, we can get the following relation
\begin{equation*}
    L^k = \sum\limits_{j\in I_{N_k}} L_j.
\end{equation*}
From Lemma \ref{le:divison_collide}, we have $L_i>0$, for $i\in I_{N_k}$. Since $L_i= min\{j\in I_{N_k}|L_j\}$, we have $L_j\ge L_i$, for $j\in I_{N_K}$. Therefore, the above formula can be rewritten as
\begin{equation*}
    L^k = \sum\limits_{j\in I_{N_k}}L_j\ge N_k\times L_i
\end{equation*}
which means that $L_i\le \frac{L_k}{N_k}$.
\end{proof}

\begin{lemma}
For agent $i$ working at layer $k$, as $t\longrightarrow \infty$, if $L_i= max\{j\in I_{N_k}|L_j\}$ and $N_k\ge 2$, we have $\frac{L^k}{N_k}\le L_i\le \frac{L^k}{2}$.
\label{le:max_l}
\end{lemma}
\begin{proof}
Similar to Lemma \ref{le:min_l}, since $L_i= max\{j\in I_{N_k}|L_j\}$, we have $L_j\le L_i$ for $j\in I_{N_K}$,
which means that $L_i\ge \frac{L_k}{N_k}$.

As shown in Figure \ref{f_lemma_coliside}, we use $l_i$ to denote the length of $E_i^3\cup E_i^4$.
We can get the following relation
\begin{equation*}
    L^k =\sum\limits_{j\in I_{N_k}}l_j.
\end{equation*}
From Lemma \ref{le:division_point} and control input (\ref{coni_s}), as $t\longrightarrow \infty$, the agent $i$ have following relation
\begin{equation*}
    L_i = 0.5(l_i+l_\alpha),
\end{equation*}
Since $l_i+l_\alpha\le L^k$, we get $L_i \le  \frac{L^k}{2} $.
\end{proof}

In our algorithm, the number of agents on the $k$-th layer is not fixed. Obviously, we can get a relation as follows $P_k(t)\ge 0$, for $t\ge 0$.
\begin{lemma}
For the layer $k$, if the agent $i$ leaves this layer and $N_k\ge 2$, the maximal reduction of the detect probability of the $k$-th layer can be calculated as follows
\begin{equation*}
\begin{aligned}
     {P_k}^\prime  &= \int_{E_i^2} {\left( {f\left( {d\left( {{\varphi _i},\theta } \right)} \right) - f\left( {d\left( {{s_i},{\varphi _i}} \right) + d\left( {{s_i},\theta } \right)} \right)} \right)\rho \left( \theta  \right)d\theta   } \\
     &+\int_{E_i^3} {\left( {f\left( {d\left( {{\varphi _i},\theta } \right)} \right) - f\left( {d\left( {{s_\beta },{\varphi _i}} \right) + d\left( {{s_\beta },\theta } \right)} \right)} \right)\rho \left( \theta  \right)d\theta }
\end{aligned}
\end{equation*}
    \label{le:delta_Pk}
\end{lemma}
\begin{proof}
    In our algorithm, when the agent $i$ enters or leaves, there is only a change in the monitoring probability of $E_i^2$ and $E_i^3$ for all regions in layer $k$.
    We assume that agent i leaves layer k at time ti, and the new division point is at the position of $\varphi_i$. From the Lemma \ref{le:division_point}, then we can get the following formula
    \begin{equation*}
    \begin{aligned}
    P_k(t_i+\varepsilon) &\ge P_k(t_i) -\int_{E_i^2} {\left( {f\left( {d\left( {{\varphi _i},\theta } \right)} \right) - f\left( {d\left( {{\varphi _\alpha },\theta } \right)} \right)} \right)\rho \left( \theta  \right)d\theta } \\
    &+ \int_{E_i^3} {\left( {f\left( {d\left( {{\varphi _i},\theta } \right)} \right) - f\left( {d\left( {{\varphi _\beta },\theta } \right)} \right)} \right)\rho \left( \theta  \right)d\theta },
    \end{aligned}
    \end{equation*}
    where $\varepsilon$ is an infinitesimal.
     From Table \ref{tab:single_layer}, we can know that $s_i$ and $s_\beta$ will converge to the midpoint of $l_\alpha$ and $l_i$.
    As shown in Fig.\ref{f_lemma_coliside}, the length of $E_i^1$ is the same as that of $E_i^2$, and the length of $E_i^3$ is the same as that of $E_i^4$.
    Therefore, the variation of the detect probability of the $k$-th layer can rewritten as follows
    \begin{equation*}
\begin{aligned}
     {P_k}^\prime  &\ge \int_{E_i^2} {\left( {f\left( {d\left( {{\varphi _i},\theta } \right)} \right) - f\left( {d\left( {{s_i},{\varphi _i}} \right) + d\left( {{s_i},\theta } \right)} \right)} \right)\rho \left( \theta  \right)d\theta   } \\
     &+\int_{E_i^3} {\left( {f\left( {d\left( {{\varphi _i},\theta } \right)} \right) - f\left( {d\left( {{s_\beta },{\varphi _i}} \right) + d\left( {{s_\beta },\theta } \right)} \right)} \right)\rho \left( \theta  \right)d\theta }   ,
\end{aligned}
\end{equation*}
which implies this lemma.
\end{proof}
\begin{lemma}
\label{le:delta_PK_increase}
    For the layer $k$, if the agent $i$ enters this layer, the minimal increase of the detect probability of the $k$-th layer can be calculated as follows
    \begin{equation*}
    \begin{aligned}
P_k^\prime  = \int_{E_i^2 \cup E_i^3} {f\left( {d\left( {{\varphi _i},\theta } \right)} \right)\rho \left( \theta  \right)d\theta }  - \int_{s_\beta ^*}^{{s_\beta }} {f\left( {d\left( {{\varphi _\beta },\theta } \right)} \right)\rho \left( \theta  \right)d\theta }  - \int_{{s_i}}^{s_\beta ^*} {f\left( {d\left( {{\varphi _\alpha },\theta } \right)} \right)\rho \left( \theta  \right)d\theta }
    \end{aligned}
    \end{equation*}
   where $s_\beta^*$ is the division point in $l_\alpha$ before agent $i$ enters layer $k$, and if $N_k=0$, then $d(\varphi_\alpha,s_i)=0$, $d(\varphi_\alpha,\theta)=0$.
\end{lemma}
\begin{proof}
    As Lemma \ref{le:delta_Pk} said, agent entry will only change the detect probabilities of $E_i^2$ and $E_i^3$ in the $k$-th layer.
    Assuming that the agent enters layer $k$ after time $t_i$, We can know the detect probability of this area as follows
\[P_k(t_i)=P_k^*(t_i)+\int_{{s_i}}^{s_\beta ^*} {f\left( {d\left( {{\varphi _\alpha },\theta } \right)} \right)\rho \left( \theta  \right)d\theta }  + \int_{s_\beta ^*}^{{s_\beta }} {f\left( {d\left( {{\varphi _\beta },\theta } \right)} \right)\rho \left( \theta  \right)d\theta } \]
where $P_k^*$ indicates that the $k$-th layer does not consider the monitoring probability of $E_i^2$ and $E_i^3$.
After the agent $i$ enters the $k$ layer, from Theorem \ref{Theorem:H} the above formula is rewritten as
\[{P_k}({t_i} + \varepsilon ) \ge P_k^*({t_i} + \varepsilon ) + \int_{E_i^2 \cup E_i^3} {f\left( {d\left( {{\varphi _i},\theta } \right)} \right)\rho \left( \theta  \right)d\theta } \]
where $P_k^*({t_i} + \varepsilon )=P_k^*({t_i})$.
Therefore, we can get $P_k^\prime$ as follows
\[P_k^\prime \ge \int_{E_i^2 \cup E_i^3} {f\left( {d\left( {{\varphi _i},\theta } \right)} \right)\rho \left( \theta  \right)d\theta }  - \int_{s_\beta ^*}^{{s_\beta }} {f\left( {d\left( {{\varphi _\beta },\theta } \right)} \right)\rho \left( \theta  \right)d\theta }  - \int_{{s_i}}^{s_\beta ^*} {f\left( {d\left( {{\varphi _\alpha },\theta } \right)} \right)\rho \left( \theta  \right)d\theta } \]
which implies this lemma.
\end{proof}
According to the above conclusions, we can get the following theorem
\begin{theorem}
For a multi-agent multi-layer barrier coverage system with $N_k$ layers, if the agent $i$ working on the $k$-th layer satisfies the following inequality,
\begin{equation*}
P_k^\prime < \frac{(1-P_k)P_v^\prime}{1-P_v-P_v^\prime}
\end{equation*}
the detection probability (\ref{eq:multi_layer_P}) of the system will increase if the agent $i$ enters the $v$-th layer.
\end{theorem}
\begin{proof}
    Without loss of generality, we can assume that when the multi-agent coverage system is at time $t_1$, agent $i$ works at layer $k$; when the system is at time $t_2$, agent $i$ works at layer $v$. And, at the two moments, except that the working place of agent $i$ is different, other agents are still working in the same layer.
    Therefore, we can get the following equation by (\ref{eq:multi_layer_P})
\begin{equation*}
    P(t_1) = 1-(1-P_1)(1-P_2)...(1-P_k)...(1-P_v)...(1-P_{N_k}).
\end{equation*}
From Lemma \ref{le:delta_Pk} and Lemma \ref{le:delta_PK_increase}, we can get the following equation
\begin{equation*}
    P(t_2) \ge 1-(1-P_1)(1-P_2)...(1-P_k+P_k\prime)...(1-P_v-P_v\prime)...(1-P_{N_k})
\end{equation*}
let $1-(1-P_1)(1-P_2)...(1-P_k+P_k\prime)...(1-P_v-P_v\prime)...(1-P_{N_k}))>P(t_1)$, we can get $P(t_1)>P(t_2)$, which implies this theorem.
Simplify the above formula to get
\begin{equation*}
    P_k^\prime < \frac{(1-P_k)P_v^\prime}{1-P_v-P_v^\prime}
\end{equation*}
This completes the proof.
\end{proof}

\begin{coro}
For a single-layer barrier coverage system with fixed division points, when the layer is a circle with a radius $R_0$, $d$ adopts the geodesic distance obtained on the layer and the detection model of the agent is a Gaussian probability model, i.e.
$f(d)=e^{-d^2/\gamma^2}$ if the radius $R_0$ satisfies $R_0\le\frac{\sqrt{2}\gamma}{2\pi}$, dynamic system (\ref{XY}) ensure that the function (\ref{H}) reaches the maximum value.
\end{coro}

\begin{proof}
By taking the partial derivative of (\ref{H}) with respect to $\varphi_i$, we get
\begin{equation*}
\frac{{\partial H}}{{\partial {\varphi _i}}} = \int_{{E_i}} {\frac{{\partial f\left( {d({\varphi _i},\theta )} \right)}}{{\partial {\varphi _i}}}\rho \left( \theta  \right)d\theta }
\end{equation*}
Substituting $f(d)=e^{-{d^2}/{\gamma^2}}$ into the above formula yields
\[\frac{{\partial H}}{{\partial {\varphi _i}}} = \int_{{E_i^1}} {{e^{ - \frac{{{d^2}}}{{{\gamma ^2}}}}}\left( { - 2\frac{d}{{{\gamma ^2}}}} \right)\frac{{\partial d}}{{\partial {\varphi _i}}}\rho \left( \theta  \right)}+\int_{{E_i^2}} {{e^{ - \frac{{{d^2}}}{{{\gamma ^2}}}}}\left( { - 2\frac{d}{{{\gamma ^2}}}} \right)\frac{{\partial d}}{{\partial {\varphi _i}}}\rho \left( \theta  \right)} \]
The integral is segmented because the geodesic distance $d$ is not derivable when $\theta=\varphi_i$. And
${\frac{{\partial d}}{{\partial {\varphi _i}}}=R_o}$ as $\theta\in E_i^2$, ${\frac{{\partial d}}{{\partial {\varphi _i}}}=-R_o}$ as $\theta\in E_i^3$.

We take the partial derivative of the above formula with respect to $\varphi_i$ to get
\[\frac{{{\partial ^2}H}}{{\partial {\varphi _i}^2}} = \int_{{E_i^2}} {{e^{ - \frac{{{d^2}}}{{{\gamma ^2}}}}}\left( {4\frac{{{d^2}}}{{{\gamma ^4}}} - \frac{2}{{{\gamma ^2}}}} \right){{\left( {\frac{{\partial d}}{{\partial {\varphi _i}}}} \right)}^2}\rho \left( \theta  \right)d\theta } +\int_{{E_i^3}} {{e^{ - \frac{{{d^2}}}{{{\gamma ^2}}}}}\left( {4\frac{{{d^2}}}{{{\gamma ^4}}} - \frac{2}{{{\gamma ^2}}}} \right){{\left( {\frac{{\partial d}}{{\partial {\varphi _i}}}} \right)}^2}\rho \left( \theta  \right)d\theta }\]
From Lemma \ref{le:max_l}, we can get $d\le\pi R_0$. If $R_0\le\frac{\sqrt{2}\gamma}{2\pi}$, we have $\frac{{{\partial ^2}H}}{{\partial {\varphi _i}^2}}<0$.
Moreover, the Hessian matrix of the function of coverage quality (\ref{H}) satisfies
\begin{equation*}
\begin{split}
{\nabla ^2}H &= [\frac{\partial^2U}{\partial \varphi_i\partial \varphi_j}]=\text{diag}(\frac{{{\partial^2}H}}{{\partial {\varphi^2_1}}}, \frac{{{\partial^2}H}}{{\partial {\varphi^2_2}}}, ... ,\frac{{{\partial ^2}H}}{{\partial {\varphi^2_N}}})\in R^{N\times N}.
\end{split}
\end{equation*}
This means $H$ has a unique maximum. Combining with Theorem \ref{Theorem:H}, the dynamic system (\ref{XY}) will ensure the function
(\ref{H}) reaches the maximum value.
\end{proof}

\section{Case Studies}\label{sec:cas}
In this section, we will give some simulation and experiment results to verify our coverage algorithm. We implemented our algorithm on MATLAB 2022a. Now, we give the multi-agent barrier coverage algorithm in Table \ref{tab:multi_layer_coverage}.

\subsection{Numerical simulation}
We designed 3 layers of area. There are 50 agents needs to cover on these three layers to monitor the invasion of intruders. These three layers are designed as follows
\begin{equation}
   \left\{ \begin{aligned}
        R_1(\theta) &= 1 + 0.15\sin(4\theta), \\
        R_2(\theta) &= 2 + 0.15\sin(10\theta),\\
        R_3(\theta) &= 3 + 0.15\sin(40\theta).
   \end{aligned}
     \right.
    \label{eq:simulation_layers}
\end{equation}
The probabilistic model is given by $f(d({\varphi _i},\theta )) = \exp ( {- {d({\varphi _i},\theta )^2}})$, where the distance function $d$ is calculated as follows
\[d({\varphi _i},\theta ) = \left\{ {\begin{array}{*{20}{c}}
{\left| {\int_{{\varphi _i}}^\theta  {\sqrt {R{{\left( \theta  \right)}^2} + R'{{\left( \theta  \right)}^2}} d\theta } } \right|,}&{if\quad\left| {\int_{{\varphi _i}}^\theta  {\sqrt {R{{\left( \theta  \right)}^2} + R'{{\left( \theta  \right)}^2}} d\theta } } \right| \le \frac{L_{k_i}}{2}}\\
{L_{k_i}- \left| {\int_{{\varphi _i}}^\theta  {\sqrt {R{{\left( \theta  \right)}^2} + R'{{\left( \theta  \right)}^2}} d\theta } } \right|.}&{otherwise}
\end{array}} \right.\]
where $d$ is Lipschitz continuous.
The density function is $ \rho(\theta) = \frac{\theta}{2{\pi}^2} $.
We set the adjustable parameters as follows
\begin{equation*}
\left\{ {\begin{array}{*{20}{c}}
{{\kappa _r} = 0.1},\\
{{\kappa _\omega } = 0.01},\\
{{\kappa _s } = 0.05}.
\end{array}} \right.
\end{equation*}

\begin{figure}[t!]
\centering\includegraphics[width=15.5cm]{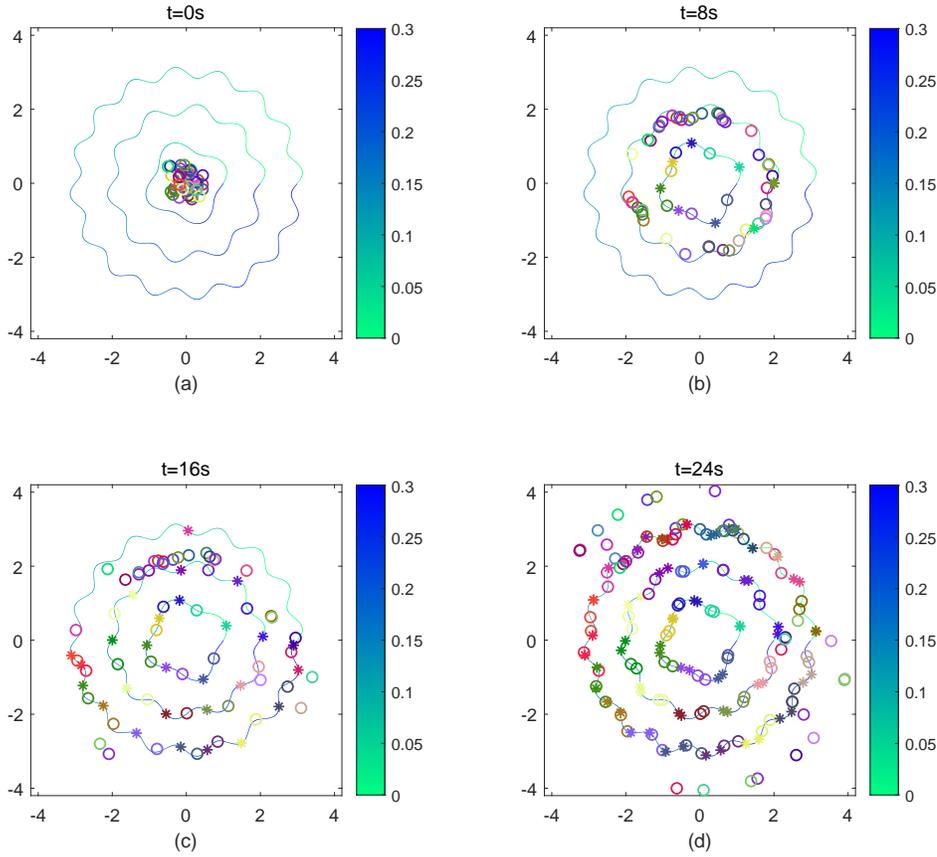}
\caption{Snapshots of simulation results. Circles denote the mobile agents, and the stars refer to the division points.}
\label{f:snapshot}
\end{figure}
As shown in Fig.\ref{f:snapshot}, we place the agent inside the innermost layer. All agents gradually expand outwards, and finally cover all three layers. And we intercept the position results of the algorithm at 4 time points, which are 0s, 8s, 16s and 24s respectively.

As shown in Fig.\ref{f:snapshot}, when the algorithm first starts running, all agents are in the innermost inner region. After the algorithm runs for 8 seconds, 6 agents have been covered on the first layer, and some agents have moved to the second layer.
Combined with Figure \ref{f:detect_probability}, after the algorithm runs for about 13 seconds, the detect probability of the third layer decreases.
When the algorithm runs to 16 seconds, we find that some agents are moving from the third layer to the second layer.
This is because the Algorithm \ref{tab:layer_change}, when the inner agent is not well qualified for its detection task, the outer agent will leave the outer layer and go to the inner layer to help the inner agent.
When the algorithm runs for 24 seconds, the multi-agent systems is basically stable, and most of the agents are already working on the layer. Around the circle with a radius of 4, some agents are patrolling, looking for any agents that need help, and when found, these patrolling agents will take action.
Finally, we give the results of the algorithm running to the last moment of the system in Figure \ref{f:final_shot}.
We can find that on each layer, the agents are denser where the invasion probability is high.
Moreover, there are still free agents patrolling the circle of radius 4.
\begin{figure}
\centering\includegraphics[width=10cm]{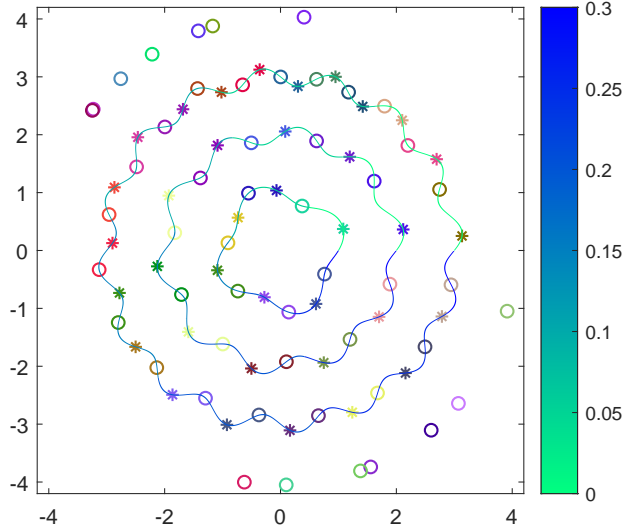}
\caption{The final result of the system state.}
\label{f:final_shot}
\end{figure}

In Fig. \ref{f:detect_probability}, we show how the detection probability of the system and each layer changes over time.
We can see that when the second and third layers have no agents, the total detection probability is the same as that of the first layer. When the second layer and the third layer have agents working one after another, the monitoring probability of the agents has a significant increase. Finally, it can be found that the detection probability of the multi-layer fence coverage algorithm exceeds 99.99$\%$.
\begin{figure}
\centering\includegraphics[width=15.5cm]{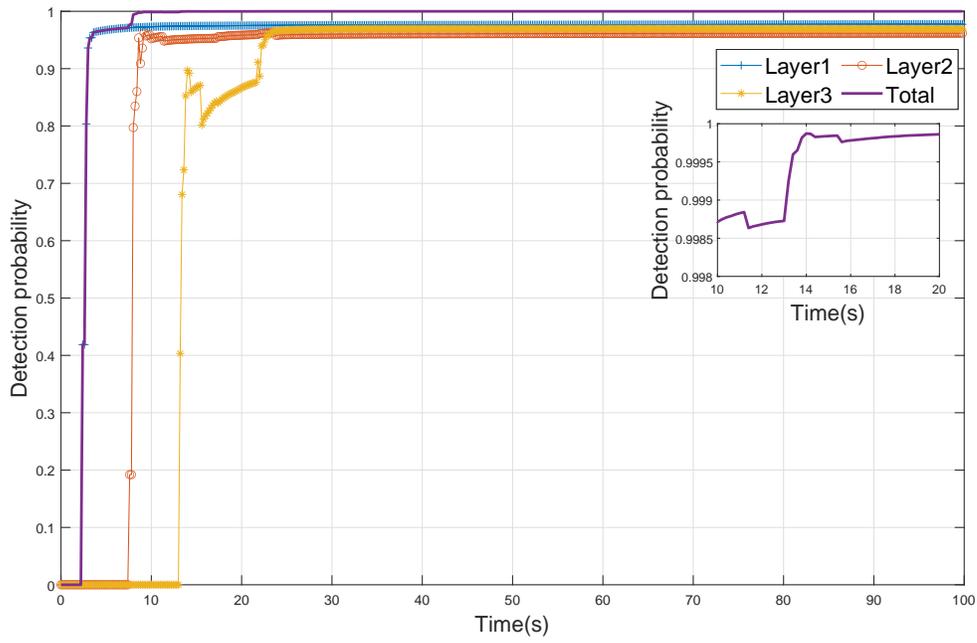}
\caption{Detection probability of each layer and total system.}
\label{f:detect_probability}
\end{figure}

We also did controlled experiments with multi-layer barrier coverage and single layer barrier coverage. As shown in Fig.\ref{f:single_and_multi}, the detection probability of the multi-layer barrier coverage was inferior to that of the single-layer fence cover for the initial period, but once agents moved to the second layer, the detection probability of the multi-layer barrier coverage reversed to that of the single-layer fence cover, and was higher than that of the single-layer for the rest of the time.
\begin{figure}
  \centering
  \includegraphics[width=15.5cm]{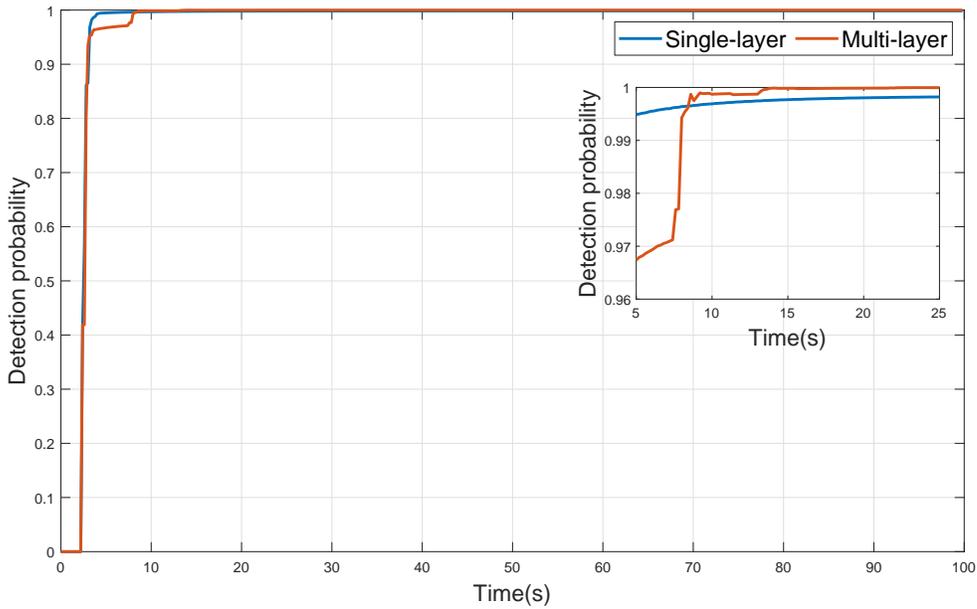}
  \caption{Difference between single layer barrier coverage and multi-layer barrier coverage for the same number of agents.}\label{f:single_and_multi}
\end{figure}

We counted the final detection probability of single-layer barrier coverage and multi-layer barrier coverage with different number of smart bodies, as shown in Fig.\ref{f:agent_number_dif}.
It can be found that there is no difference in the detect probability between single and multi-layer barrier coverage when the number of agents is small. However, the detection probability of the multi-layer barrier coverage is significantly higher than that of the single-layer barrier coverage when the number of agents gradually increases. When the number of smart bodies is large enough, the increase in the number of smart bodies is of little help to the single-layer barrier coverage.
When the number of agents is 50, the detection probability of single-layer barrier coverage reaches 99.8 percent, while the detection probability of multi-layer barrier coverage is very close to 100 percent.
\begin{figure}[t!]
\centering\includegraphics[width=15.5cm]{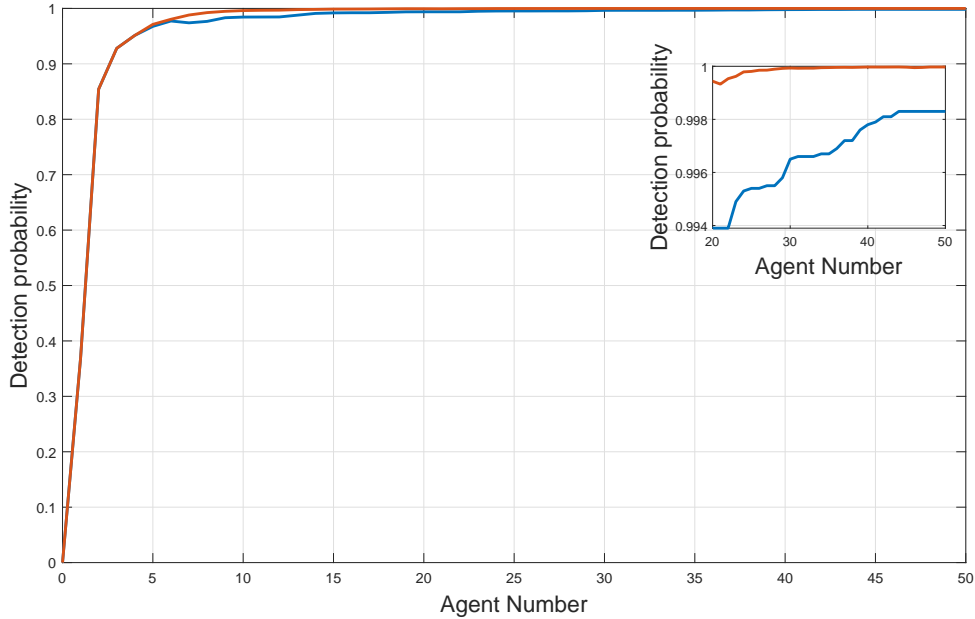}
\caption{Difference in detection probability between single and multi-layer barrier coverage for the same number of agents.}
\label{f:agent_number_dif}
\end{figure}

\section{Conclusions}\label{sec:con}

This paper presented a distributed multi-agent barrier coverage algorithm. First, a single-layer barrier coverage quality function was designed based on the probabilistic model of intrusion and a single-layer barrier coverage algorithm was designed based on the gradient method. Then a layer-to-layer adjustment mechanism was proposed based on the single-layer algorithm, which adjusts the number of agents on each layer so that the coverage quality of the whole system was improved. Then some theoretical analyses were given to theoretically verify the stability and effectiveness of the single-layer algorithm and the necessity of the multi-layer algorithm, and the theoretical results were given in some special cases. Finally, the effectiveness of our algorithm was verified by simulation and the practicality of the algorithm was verified by experiment.

\section{Appendix}

\section*{Acknowledgment}
The Project was supported by the Fundamental Research Funds for the Central Universities, China University of Geosciences~(Wuhan).


\begin{thebibliography}{99}
\bibitem{VT1995}Vicsek, Tamás, et al. "Novel type of phase transition in a system of self-driven particles." Physical review letters 75.6 (1995): 1226.

\bibitem{ro2020}Wilson S, Glotfelter P, Wang L, et al. The robotarium: Globally impactful opportunities, challenges, and lessons learned in remote-access, distributed control of multirobot systems[J]. IEEE Control Systems Magazine, 2020, 40(1): 26-44.

\bibitem{CJ2004}	Cortes J, Martinez S, Karatas T, et al. Coverage control for mobile sensing networks[J]. IEEE Transactions on robotics and Automation, 2004, 20(2): 243-255.

\bibitem{THM2013}   	Thanou M, Stergiopoulos Y, Tzes A. Distributed coverage using geodesic metric for non-convex environments[C]. 2013 IEEE international conference on robotics and automation. IEEE, 2013: 933-938.

\bibitem{zhaim2021} Zhai C, Zhang H T, Xiao G. Cooperative Coverage Control of Multi-Agent Systems and its Applications[M]. Springer, 2021.

\bibitem{SM2019}    Santos M, Mayya S, Notomista G, et al. Decentralized minimum-energy coverage control for time-varying density functions[C]. 2019 International Symposium on Multi-Robot and Multi-Agent Systems (MRS). IEEE, 2019: 155-161.

\bibitem{BA2020}	Benevento A, Santos M, Notarstefano G, et al. Multi-robot coordination for estimation and coverage of unknown spatial fields[C]. 2020 IEEE International Conference on Robotics and Automation (ICRA). IEEE, 2020: 7740-7746.

\bibitem{zhai13}	Zhai C, Hong Y. Decentralized sweep coverage algorithm for multi-agent systems with workload uncertainties[J]. Automatica, 2013, 49(7): 2154-2159.

\bibitem{IS2016}	Ivić S, Crnković B, Mezić I. Ergodicity-based cooperative multiagent area coverage via a potential field[J]. IEEE transactions on cybernetics, 2016, 47(8): 1983-1993.

\bibitem{Zheng2022}Zheng Y, Zhai C. Distributed Coverage Control of Multi-Agent Systems in Uncertain Environments using Heat Transfer Equations[J]. arXiv preprint arXiv:2204.09289, 2022.


\bibitem{ks2005}	Kumar S, Lai T H, Arora A. Barrier coverage with wireless sensors[C]//Proceedings of the 11th annual international conference on Mobile computing and networking. 2005: 284-298.

\bibitem{liu2008}    Liu B, Dousse O, Wang J, et al. Strong barrier coverage of wireless sensor networks[C]//Proceedings of the 9th ACM international symposium on Mobile ad hoc networking and computing. 2008: 411-420.

 \bibitem{wang2013}   Wang Z, Liao J, Cao Q, et al. Achieving k-barrier coverage in hybrid directional sensor networks[J]. IEEE Transactions on Mobile Computing, 2013, 13(7): 1443-1455.

\bibitem{chen2007}    Chen A, Kumar S, Lai T H. Designing localized algorithms for barrier coverage[C]//Proceedings of the 13th annual ACM international conference on Mobile computing and networking. 2007: 63-74.
\bibitem{cheng2009}	Cheng T M, Savkin A V. A distributed self-deployment algorithm for the coverage of mobile wireless sensor networks[J]. IEEE Communications Letters, 2009, 13(11): 877-879.
 \bibitem{banD2010}   Ban D, Jiang J, Yang W, et al. Strong k-barrier coverage with mobile sensors[C]//Proceedings of the 6th International Wireless Communications and Mobile Computing Conference. 2010: 68-72.

 \bibitem{BB2009}    Bhattacharya B, Burmester M, Hu Y, et al. Optimal movement of mobile sensors for barrier coverage of a planar region[J]. Theoretical Computer Science, 2009, 410(52): 5515-5528.

       \bibitem{song2018}	Song C, Fan Y. Coverage control for mobile sensor networks with limited communication ranges on a circle[J]. Automatica, 2018, 92: 155-161.

\bibitem{zhai16} C. Zhai, F. He, Y. Hong, L. Wang and Y. Yao, Coverage-based interception algorithm of multiple interceptors against the target involving decoys. {\sl AIAA Journal of Guidance, Control, and Dynamics}, pp.1-7, 2016.





\end{thebibliography}
\end{document}